\documentclass[graybox]{svmult}

\usepackage{type1cm}        % activate if the above 3 fonts are
\usepackage{makeidx}         % allows index generation
\usepackage{graphicx}        % standard LaTeX graphics tool
\usepackage{multicol}        % used for the two-column index
\usepackage[bottom]{footmisc}% places footnotes at page bottom

\usepackage{newtxtext}       % 
\usepackage{newtxmath}       % selects Times Roman as basic font

\usepackage{enumitem}

\def\bR{\mathbb{R}}
\def\bC{\mathbb{C}}
\def\bS{\mathbb{S}}
\def\bN{\mathbb{N}}
\def\bZ{\mathbb{Z}}

\def\cA{\mathcal{A}}
\def\cF{\mathcal{F}}
\def\cS{\mathcal{S}}

\def\cG{\mathcal{G}}
\def\cB{\mathcal{B}}

\def\cH{\mathcal{H}}
\def\cD{\mathcal{D}}

\def\vp{\varphi}

\def\rn{\bR^n}
\def\rno{\bR^n\setminus\{0\}}

\def\zn{\bZ^n}
\def\rnn{\bR^{2n}}
\def\rnno{\bR^{2n}\setminus\{0\}}
\def\znn{\bZ^{2n}}
\def\la{\langle}
\def\ra{\rangle}
\def\lc{\left(}
\def\rc{\right)}
\def\lV{\left\lVert}
\def\rV{\right\rVert}

\def\supp{\mathrm{supp}}
\def\wt{\widetilde}
\def\*b{*_{\bullet}}
\def\w{\mathrm{w}}

\newcommand{\SO}{S^0_{0,0}}
\newcommand{\Gz}{\Gamma_{z_0}}

\def\Bd'{B_{\delta'}}

\def\cBd'{\bar{B}_{\delta'}}

\newcommand{\GLL}{\mathrm{GL}\left(2n,\mathbb{R}\right)}

\def\Spdr{\mathrm{Sp}(n,\bR)}

\def\smo{\setminus\{0\}}
\newcommand{\Sm}{S^m_{0,0}}
\def\opt{\mathrm{Op}_\tau}
\def\cones{\mathrm{conesupp}}
\newcommand{\chr}{\mathrm{char}}
\newcommand{\s}{_{v_s}}
\newcommand{\mpqsn}{M^{p,q}\s(\rn)}

\newcommand{\vpt}{\widetilde{\varphi}}
\newcommand{\sumin}{\sum_{\lambda \in \Lambda\cap \Gamma_{z_0}}}
\newcommand{\sumo}{\sum_{\lambda \in \Lambda\setminus \Gamma_{z_0}}}

\makeindex             % used for the subject index
\begin{document}

\title*{An introduction to the Gabor wave front set}
\author{Luigi Rodino and S. Ivan Trapasso}
\institute{Luigi Rodino \at Dipartimento di Matematica, Universit\`a di Torino, Via Carlo Alberto 10, 10123, Torino. \email{luigi.rodino@unito.it}
\and S. Ivan Trapasso \at Dipartimento di Scienze Matematiche (DISMA) ``G. L. Lagrange'', Politecnico di Torino, Corso Duca degli Abruzzi 24, 10129, Torino. \email {salvatore.trapasso@polito.it}}
\maketitle

\abstract{In this expository note we present an introduction to the Gabor wave front set. As is often the case, this tool in microlocal analysis has been introduced and reinvented in different forms which turn out to be equivalent or intimately related. We provide a short review of the history of this notion and then focus on some recent variations inspired by function spaces in time-frequency analysis. Old and new results are presented, together with a number of concrete examples and applications to the problem of propagation of singularities. }

\section{Introduction} A central notion in microlocal analysis of partial differential equations is the wave front set \cite{hormander book 1}. In somewhat rough terms, the wave front set of a distribution $u$ is the collection of all the points of the phase space $(x_0,\xi_0)$, $\xi_0 \ne 0$, where the lack of regularity of $u$ at $x_0$ is detected on the spectral side by a characteristic behaviour in the direction $\xi_0$. Giving a rigorous meaning to this heuristic model provides a fine scale of technical tools for the microlocal study of singularities of pseudodifferential operators and their propagation. It should be stressed that wave front sets play a major role in the mathematical theory of quantum fields \cite{duis,radz}. We cannot frame here the long tradition of studies on the wave front set and its applications; a complete historical and technical account may be found in the monograph \cite{hormander book 1} by H\"ormander, who first introduced wave front sets in \cite{hormander first}. 

In recent times the notion of wave front set have benefited from the perspective of time-frequency analysis \cite{johansson,pil tt1,pil tt2,rodino wahlberg}. The spirit of Gabor analysis may be condensed in the idea of simultaneous analysis of distributions with respect to both time and frequency variables; several techniques and function spaces were introduced in the last decades to carry out this program \cite{gro book}. The affinities with the notion of wave front set, where the regularity is measured by a simultaneous analysis of points and directions, are evident.  

The purpose of this introductory paper is to present some of the contributions in this respect, in particular we focus on the Gabor wave front set \cite{rodino wahlberg}. To be precise, the idea of a \textit{global} wave front set showed up many times under several different guises; an historical account on the issue with many pointers to the literature is given in Section 3, while in Section 2 we collected some preliminary material from microlocal and time-frequency analysis. 

In Section 4 we provide a more technical description of the Gabor wave front set. In particular we highlight the most important results of \cite{CNR prop}. \cite{hormander quadratic} and  \cite{rodino wahlberg}, together with a number of detailed examples. New results for the wave front set in context of modulation space regularity are derived in Section 4.3. We conclude  with a brief review of applications to propagation of singularities. 

Most of the technical proofs are omitted to keep the presentation at an introductory level. We hope that this overview may be useful as a point of departure for the interested reader, as well as a practical summation of the most relevant results on the topic.  

\section{Preliminaries} 
\subsection{Notation} We set $x^2=x\cdot x$, for $x\in\rn$, where $x\cdot y = xy$ is the scalar product on $\rn$. The Schwartz class is denoted by  $\cS(\rn)$, the space of temperate distributions by $\cS'(\rn)$. The brackets $\langle  f,g\rangle $ denote the extension to $\cS' (\rn)\times \cS (\rn)$ of the inner product $\langle f,g\rangle=\int_{\rn} f(x){\overline {g(x)}}dx$ on $L^2(\rn)$.

The conjugate exponent $p'$ of $p \in [1,\infty]$ is defined by $1/p+1/p'=1$. The symbol $\lesssim$ means that the underlying inequality holds up to a positive constant factor $C>0$.
For any $x\in\rn$ and $s\in\bR$
we set $\la x \ra^{s}\coloneqq(1+\left|x\right|^{2})^{s/2}$.
We choose the following normalization for the Fourier transform:
\[
\hat{f}(\xi) = \cF f(\xi) =\int_{\rn}e^{-2\pi ix\xi}f(x) d x,\qquad\xi\in \rn. \]
We define the translation and modulation operators: for
any $x,\xi \in\rn $ and $f\in\cS(\rn)$,
\[
\left(T_{x}f\right)\left(y\right)\coloneqq f(y-x),\qquad\left(M_{\xi}f\right)(y)\coloneqq e^{2\pi i \xi y}f(y).
\]
These operators can be extended by duality on temperate distributions. The composition $\pi(x,\xi)=M_\xi T_x$ constitutes a so-called \textit{time-frequency shift}.

Recall that $\Gamma \subset \rn$ is a \textit{conic subset} of $\rn$ if it is invariant under multiplication by positive real numbers, namely $x \in \Gamma \Rightarrow \lambda x \in \Gamma$ for any $\lambda >0$. 

The \textit{symplectic group} $\Spdr$ consists of all $2n\times 2n$ invertible matrices $S \in \GLL$ such that 
\[ S^\top J S = S J S^\top = J, \quad J= \lc \begin{array}{cc} O & I \\ -I & O \end{array}\rc, \] where $J$ is the canonical symplectic matrix and $O$ and $I$ are the $n\times n$ zero and identity matrices respectively. 

In the rest of the paper we identify the cotangent set $T^*\rn$ of $\rn$ with $\rnn$ to lighten the notation.

\subsection{Modulation spaces}\label{sec mod} The short-time Fourier transform (STFT) of a temperate distribution $u\in\cS'(\rn)$ with respect to the window function $\vp \in \cS(\rn)\setminus\{0\}$ is defined by
\begin{equation}\label{FTdef}
V_\vp u (x,\xi)\coloneqq \cF (u\cdot T_x \vp)(\xi)=\int_{\rn}e^{-2\pi iy  \xi } u(y)\, {\overline {\vp(y-x)}}\,dy.
\end{equation}

The reader may want to consult the monograph \cite{gro book} for a comprehensive treatment of the mathematical properties of this time-frequency representation, in particular those mentioned below. We highlight that the STFT is intimately connected with other well-known phase-space transforms, in particular the Wigner distribution
\begin{equation} W(u,\vp)(x,\xi )=\int_{\mathbb{R}^{d}}e^{-2\pi iy \xi }u\left(x+\frac{y}{2}\right)%
\overline{\vp\left(x-\frac{y}{2}\right)}\ dy. \label{def wig} \end{equation}

As far as the regularity is concerned, the STFT of a possibly wild distribution $u \in \cS'(\rn)$ is a well-behaved function; in particular, we have that $V_\vp u \in C(\rnn)$ and there exists constant $C>0$ and $N \ge0$ such that $|V_\vp u (z)| \le C \la z \ra ^N$ for all $z \in \rnn$. Furthermore, $V_\vp u \in \cS(\rnn) \Leftrightarrow u \in \cS(\rn)$. It turns out that the STFT is one-to-one in $\cS'(\rn)$, as a result of the following \textit{inversion formula}: for $u \in \cS'(\rn)$ and $\vp,\psi \in \cS(\rn)\smo$ such that $\la \vp,\psi \ra \ne 0$ we have
\begin{equation}\label{rec form cont}
u = \frac{1}{\la \vp,\psi \ra}  \int_{\rnn} V_\vp u (z) \pi(z)\psi dz, \end{equation} to be interpreted in the distribution sense - namely, the right-hand side is a temperate distribution whose action on $\phi \in \cS(\rn)$ coincides with $\la u,\phi\ra$. Notice in particular that if we choose $\vp \in \cS(\rn)\smo$ with $\lV \vp \rV_{L^2} = 1$ and set $\psi=\vp$ we have
\begin{equation}\label{stft rec cont}
|V_\vp u(w)| = \left| \int_{\rnn} V_\vp u(z) V_\vp \vp (w-z)dz\right|, \quad w \in \rnn. \end{equation}
This argument generalizes to the following pointwise inequality (``change-of-window lemma'') which will be used below. 
\begin{lemma}[{\cite[Lem. 11.3.3]{gro book}}]\label{change of w}
	Let $\vp_1,\vp_2,\phi \in \cS(\rn)$ be such that $\la \phi,\vp_1 \ra \ne 0$ and $u \in \cS'(\rn)$. Therefore 
	\[ |V_{\vp_2}u (x,\xi)| \le \frac{1}{|\la \phi,\vp_1 \ra|}(|V_{\vp_1}u|*|V_{\vp_2}\phi|)(x,\xi), \quad \forall (x,\xi) \in \rnn. \]
\end{lemma} 

When speaking of weight functions below we refer to some positive function $m\in L^\infty_{\mathrm{loc}}(\rnn)$ such that $m(z+\zeta) \lesssim m(z)\la \zeta \ra^r$ for some $r\ge 0$ and any $z,\zeta \in \rnn$ - that is, $m$ is $\la \cdot \ra ^r$-moderate.

Given a non-zero window $\vp\in\cS(\rn)$, a weight function $m$ on $\rnn$ and $1\leq p,q\leq \infty$, the {\it modulation space} $M^{p,q}_m(\rn)$ consists of all temperate distributions $u\in\cS'(\rn)$ such that $V_\vp u\in L^{p,q}_m(\rnn)$ (mixed weighted Lebesgue space), that is: 
\[ \|u\|_{M^{p,q}_m}=\|V_\vp u\|_{L^{p,q}_s}=\left(\int_{\rn}
\left(\int_{\rn}|V_\vp u(x,\xi)|^pm(x,\xi)^p\,
dx\right)^{q/p} d\xi \right)^{1/q}  \, <\infty ,\]
with trivial modification if $p$ or $q$ is $\infty$. 
If $p=q$, we write $M^p$ instead of $M^{p,p}$, while for the unweighted case ($m=0$) we set $M_{0}^{p,q}\equiv M^{p,q}$. 

It can be proved that $M^{p,q}_m(\rn)$ is a Banach space whose definition does not depend on the choice of the window $\vp$ (in the sense that different windows yield equivalent norms). The standard weight used in the rest of the paper is $m(z) = v_s(z) = \la z \ra^s$ for some $s \in \bR$. 
We mention that many common function spaces are intimately related with modulation spaces: for instance, $M^2(\rn)$ coincides with the Hilbert space $L^2(\rn)$, while if $m(x,\xi)= \la \xi \ra^s$ for $s\ge0$ we have that $M^2_m(\rn)$ coincides with the usual $L^2$-based Sobolev space $H^s(\rn)$. Furthermore, the following characterizations hold for any $1\le p,q \le \infty$:
\begin{equation}\label{eq cS mod}
\cS(\rn) = \bigcap_{s\ge0} \mpqsn, \quad \cS'(\rn)= \bigcup_{s\ge0} M^{p,q}_{v_{-s}}(\rn). 
\end{equation}

Another perspective on modulation spaces is provided by inspecting the definition of the STFT $V_\vp u$: it may be thought of as a continuous expansion of the function $u$ with respect to the uncountable system $\{\pi(z)\vp : z=(x,\xi) \in \rnn \}$. Notice that $\pi(z)\vp$ is a wave packet highly concentrated near $z$ in phase space. For short, we have $V_\vp u(x,\xi) = \la u,\pi(x,\xi)\vp \ra$ in the sense of the (extension to the duality $\cS'-\cS$ of the) inner product on $L^2$. This perspective is further reinforced by the role of \textit{frame theory} and discrete time-frequency representations. Given a non-zero window function $\vp \in L^2(\rn)$ and a subset $\Lambda \subset \rnn$, we say that the collection of the time-frequency shifts of $\vp$ along $\Lambda$ is a \textit{Gabor system}, namely \[ \cG(\vp,\Lambda) = \{ \pi(z)\vp \, : \, z \in \Lambda\}.  \] For instance one may consider separable lattices such as \[ \Lambda = \alpha \bZ \times \beta \bZ = \{ (\alpha k,\beta n) \, : \, k,n \in \bN \}, \] for lattice parameters $\alpha,\beta >0$; we write $\cG(g,\alpha,\beta)$ for the corresponding Gabor system. 
Recall that a \textit{frame} for a Hilbert space $\cH$ is a sequence $\{x_j\}_{j\in J} \subset \cH$ such that for all $x \in \cH$ \[ A\lV x \rV_{\cH}^2 \le \sum_{j\in J} |\la x,x_j\ra|^2 \le B\lV x \rV_{\cH}^2, \] for some universal constants $A,B>0$ (frame bounds). Roughly speaking, the paradigm of frame theory consists in decomposing a vector $x$ along the frame, then studying the action of operators on such elementary pieces and finally reconstructing the image vector. The entire process is encoded by the \textit{frame operator}
\[ S \, : \, \cH \ni x \mapsto \sum_{j\in J}\la x,x_j\ra x_j \in \cH. \]
If a Gabor system $\cG(\vp,\Lambda)$ is a frame for $L^2(\rn)$ it is called \textit{Gabor frame}. Notice that the Gabor frame operator reads 
\[ Sf = \sum_{z\in \Lambda} V_gf(z)\pi(z)g, \] and is a positive, bounded invertible operator on $L^2(\rn)$.  
A remarkable result of frame theory is that a function can be reconstructed from its Gabor coefficients by means of the following discrete analogue of \eqref{rec form cont}: 
\begin{equation}\label{rec form disc}
u = \sum_{z\in \Lambda} V_{\vp}u(z)\pi(z)\vpt,  \end{equation} where $\vpt = S^{-1}\vp$ is the \textit{canonical dual window} and the sum is unconditionally convergent in $L^2$. Notice that $\vp \in \cS(\rn) \Rightarrow \vpt \in \cS(\rn)$ if $\cG(\vp,\Lambda)$ is a Gabor frame \cite{jans}.   

Moreover, the reconstruction formula \eqref{rec form disc} extends to $u \in M^{p,q}_m(\rn)$ for all $1\le p,q \le \infty$ and weight function $m$ on $\rnn$, with unconditional convergence in the modulation space norm if $1\le p,q < \infty$ (weak unconditional otherwise). In addition, an equivalent discrete norm for $M^{p,q}_m(\rn)$ is given by 
\[ \lV u \rV_{M^{p,q}_s} = \lc \sum_{n\in \zn} \lc \sum_{k \in \zn} |V_\vp u(\alpha k, \beta n)m(\alpha k, \beta n)|^p \rc^{q/p} \rc^{1/q}. \]

\subsection{Pseudodifferential operators} In the spirit of time-frequency analysis we define Weyl operators starting from the relation
\begin{equation}
\langle \sigma^{\mathrm{w}}f,g\rangle =\langle \sigma,W(g,f)\rangle ,\qquad\forall f,g\in\mathcal{S}(\mathbb{R}^{d}),\label{def wig dual}
\end{equation}
where $\sigma\in\mathcal{S}'(\mathbb{R}^{2d})$ is the \textit{symbol} of the \textit{Weyl operator} $\sigma^{\mathrm{w}}:\mathcal{S}(\mathbb{R}^{d})\rightarrow\mathcal{S}'(\mathbb{R}^{d})$, which can be formally represented as
\[
\sigma^{\text{w}}f\left(x\right) = \int_{\mathbb{R}^{2d}}e^{2\pi  i\left(x-y\right) \xi}\sigma\left(\frac{x+y}{2},\xi\right)f(y)dyd\xi,
\]
while $W(g,f)$ is the Wigner transform defined in \eqref{def wig}. Other quantization rules may be covered in a similar fashion. In particular, we define 
\begin{equation}
\langle \opt(\sigma)f,g\rangle =\langle \sigma,W_\tau(g,f)\rangle ,\qquad\forall f,g\in\mathcal{S}(\mathbb{R}^{d}),\label{def wig tau}
\end{equation}
where the Wigner distribution is generalized as
\[ W_\tau(f,g)(x,\xi )=\int_{\mathbb{R}^{d}}e^{-2\pi iy \xi }f(x+\tau y) \overline{g(x-(1-\tau)y)}\ dy. \] We refer to the papers \cite{bdo,CDT,CNT,CT} for results in this general framework. Notice that we recapture the Weyl quantization for $\tau=1/2$, while the case $\tau=0$ corresponds to Kohn-Nirenberg quantization. In the rest of the paper we will focus on Weyl operators, but most of the stated results can be transferred to other kind of pseudodifferential operators in view of the identity
\begin{equation}\label{transfer tau} \mathrm{Op}_{\tau_1}(a) = \mathrm{Op}_{\tau_2}(T_{\tau_1,\tau_2}a), \quad T_{\tau_1,\tau_2}a=e^{2\pi i (\tau_1 - \tau_2) D_x D_\xi}a, \quad a \in \cS'(\rnn). \end{equation} 

Nevertheless, there is a distinctive property characterizing the Weyl calculus among other quantization rules, which is known as \textit{symplectic covariance}. Recall indeed that $S \in \Spdr$ can be associated with a unitary bounded operator $\mu(S)$ on $L^2(\rn)$, called \textit{metaplectic operator}, which satisfies the intertwining property
\[ \mu(S)^{-1} \sigma^\w \mu(S) = (\sigma\circ S)^\w, \quad \sigma \in \cS'(\rnn). \] This shows that the map $\mu: S \mapsto \mu(S)$ defines a metaplectic operator only up to a constant complex factor of modulus one. We will not focus on technical details concerning the metaplectic representation; in fact, we have been quite sloppy in describing these features. The reader may consult \cite{dg symp, gro book,wong} for a precise account on symplectic covariance and metaplectic operators.  

A major advantage of the time-frequency analysis perspective on pseudodifferential operators is that general symbol classes may be considered, in particular modulation spaces. Recall the definition of the classical H\"ordmander classes \cite{hormander book 3}.
\begin{definition} Let $m \in \bR$. The symbol class $\Sm$ is the subspace of smooth functions $a \in C^{\infty}(\rnn)$ such that 
	\[ \sup_{(x,\xi) \in \rnn} \la \xi \ra^{-m} |\partial^\alpha_x \partial^\beta_\xi a(x,\xi)| < \infty, \quad \forall \alpha,\beta \in \bN^n_0.  \]
It is a Fr\'echet space with the obvious seminorms.
\end{definition}

For $a \in \Sm$ we have that $a^\w$ is continuous on $\cS(\rn)$ and $\cS'(\rn)$; moreover the map $T_{0,1/2}$ is an automorphism of $\Sm$. Composition of Weyl operators with symbols in $\Sm$ classes is well behaved: if $a \in \Sm$ and $b \in S^n_{0,0}$, then $a^\w \circ b^\w$ is again a Weyl operator with symbol $a\#b \in S^{m+n}_{0,0}$ - the latter is known as the \textit{Weyl (or twisted) product} of $a$ and $b$. While explicit formulas are known for $a\#b$ in general, we stress that the calculus associated with symbols in $\Sm$ is highly non-trivial due to the lack of asymptotic expansions for Weyl product of symbols. 

A somewhat better behaviour is showed by \textit{Shubin symbol classes} \cite{shubin}, defined as follows. 
\begin{definition}
	Let $m \in \bR$. The symbol class $G^m$ is the subspace of smooth functions $a \in C^{\infty}(\rnn)$ such that 
	\[ \sup_{z \in \rnn} \la z \ra^{-m+|\alpha|} |\partial^\alpha_z a(z)| < \infty, \quad \forall \alpha \in \bN^{2n}_0.  \]
	It is a Fr\'echet space with the obvious seminorms.
\end{definition}
We confine ourselves to recall that $\bigcap_{m\in \bR}G^m = \cS(\rnn)$ and the Weyl product is a bilinear continuous map $\# : G^m \times G^n \to G^{m+n}$. We also set $G^\infty = \bigcup_{m\in \bR}G^m$. 

\section{A short history of the Gabor wave front set} By analogy with the classical Huygens construction of a propagating wave, H\"ormander (\cite{hormander FIO}, 1971) called \textit{wave front set} of a distribution $u$ the subset $WF(u)$ of $\rn_x \times (\rn_\xi \setminus\{0\})$ defined by examining the behaviour at infinity of the Fourier transform $\hat{u}$. Namely, the point $(x_0,\xi_0)$, $\xi_0 \ne 0$, does \textit{not} belong to $WF(u)$ if there exist a function $\varphi \in C^\infty_c(\rn)$, $\varphi(x_0)\ne 0$, and a conic neighbourhood $\Gamma_{\xi_0}\subset \rn$ of $\xi_0$ such that
\begin{equation}
|\widehat{\varphi u}(\xi)| \le C_N \la \xi \ra^{-N} \quad \forall \xi \in \Gamma_{\xi_0}, \, N \in \bN,
\end{equation} for a suitable constant $C_N >0$. Here and below we assume $u \in \cS'(\rn)$, though the preceding estimate applies obviously to $u \in \cD'(\rn)$ or $u \in \cD'(\Omega)$ with $\Omega$ open subset of $\rn$ such that $x_0 \in \Omega$ and $\supp(\varphi) \subset \Omega$. 

An alternative definition can be given by using classical pseudodifferential operators with polyhomogeneous symbol with respect to the $\xi$ variables:
\begin{equation}
p(x,\xi) = p_m(x,\xi) + \ldots,
\end{equation} where $p_m$ satisfies $p_m(x,\lambda \xi) = \lambda^m p(x,\xi)$ for $\lambda >0$ and $\xi \ne 0$. Precisely, $(x_0,\xi_0) \notin WF(u)$ if and only if there exists $p(x,\xi)$ with $p_m(x_0,\xi_0) \ne 0$ such that $p(x,D)u \in C^\infty(U_{x_0})$ for some neighbourhood $U_{x_0}$ of $x_0$. The statement does not depend on the quantization rule we adopt to define $p(x,D)$. 

Afterwards, several variables of the definition of $WF(u)$ appeared. Our attention is focused here on the \textit{global wave front set} of H\"ormander (\cite{hormander quadratic}, 1989), which we denote here by $WF_G(u)$. To define $WF_G(u)$ for $u\in \cS'(\rn)$ we may imitate the preceding argument in terms of pseudodifferential operators, by taking now polyhomogeneous symbols in the $z=(x,\xi)$ variable, as in Shubin \cite{shubin}:
\begin{equation}
p(z)=p_m(z)+\ldots,
\end{equation} with $p_m(\lambda z) = \lambda^m p_m(z)$ for $\lambda >0$, and similarly for lower order terms. Then, $z_0=(x_0,\xi_0) \notin WF_G(u)$, $z_0 \notin 0$, if there exists $p(z)$ with $p_m(z_0)\ne 0$ such that $p(x,D)u \in \cS(\rn)$. Willing to give a direct definition, we may replace the Fourier transform with the integral transformation 
\begin{equation}\label{def Tu}
Tu(x,\xi)=\int_{\rn} e^{-2\pi it\xi}e^{-|t-x|^2/2}u(t)dt.
\end{equation} We have that $z_0 =(x_0,\xi_0) \notin WF_G(u)$ if and only if there exists a conic neighbourhood $\Gamma_{z_0}$ of $z_0$ in $\rnn$ such that 
\begin{equation}\label{Tu decay}
|Tu(z)|\le C_N \la z \ra^{-N}, \quad \forall z \in \Gamma_{z_0}, \, N \in \bN.
\end{equation}
In the next sections we shall review the main properties of $WF_G(u)$ and present some variants of the definition. We continue here by listing some papers of the last thirty years, where $WF_G(u)$ was reinvented, without reference to the original contribution by H\"ormander \cite{hormander quadratic}.

Let us first mention Nakamura (\cite{nakamura}, 2005), who introduced the so-called \textit{homogeneous wave front set} to study propagation of singularities for Schr\"odinger through methods typically used in semiclassical analysis. Schulz and Wahlberg (\cite{schulz wahlberg}, 2017) proved recently that the homogeneous wave front set coincides with $WF_G(u)$. In turn, Ito (\cite{ito}, 2006) clarified the connection of the homogeneous wave front set with the \textit{quadratic scattering wave front set} of Wunsch (\cite{wunsch}, 1999), see also \cite{mvw}. 

To complete this survey, we may mention the related definition of the \textit{scattering wave front set} of Melrose \cite{melrose}, Melrose and Zworski \cite{mz}, coinciding with the \textit{SG wave front set} of Cordes \cite{cordes} and Coriasco and Maniccia \cite{cor man}.

Roughly speaking, the scattering/SG wave front set consists of three components: $WF(u)$, $WF(\hat{u})$ and a third component similar to $WF_G(u)$ where analysis is limited to rays through $z_0=(x_0,\xi_0)$, with $x_0\in \bS^{n-1}_x$ and $\xi_0 \in \bS^{n-1}_\xi$. The enormous developments of the corresponding SG-microlocal analysis are somewhat outside our present perspective, see for instance \cite{cor k t} for references.

A new approach to $WF_G(u)$ was proposed by Rodino and Wahlberg (\cite{rodino wahlberg}, 2014) where the original contribution by H\"ormander \cite{hormander quadratic} was finally recognized and a further equivalent definition was given in terms of time-frequency analysis. Namely, the integral transform in \eqref{def Tu} coincides with the Bargmann-Gabor transform of $u$, that is a short-time Fourier transform with Gaussian window, see \cite{gabor} and the textbook \cite{gro book}. It is then natural to replace $Tu$ with the discrete Gabor frame representation of $u$, possibly with more general windows, and impose in the cone $\Gamma_{z_0}$ a rapid decay of the Gabor coefficients, see the next section for the details. In \cite{rodino wahlberg} the authors gave the name \textit{Gabor wave front set} to the associated wave front set and introduced the notation $WF_G(u)$, where the subscript $G$ stands both for global and Gabor. 

In these last five years, this new approach and the new name were adopted by a number of authors working in the area of time-frequency analysis. Let us try to give a short account. As already evident from the original work of H\"ormander \cite{hormander quadratic}, the main application concerns the propagation of microlocal singularities for the Schr\"odinger equation
\begin{equation}
\begin{cases} i\partial_t u(t,x) = H(x,D)u(t,x) \\ u(0,x) = u_0(x) \end{cases}. 
\end{equation}
A basic example is the quantum harmonic oscillator, corresponding to the Hamiltonian $H(x,\xi)=|x|^2 + |\xi|^2$. In fact, starting from the Gabor-Fourier integral representation of the Schr\"odinger propagator in \cite{CGNR,CNR spars} one can deduce in a natural way propagation in terms of $WF_G(u)$, see \cite{CNR survey,CNR prop, CNR wp,CN pot mod}. The analysis is extended to the case of non-self-adjoint Hamiltonians in \cite{car wahl,nicola semil, pravda rw, pravda mehler, wahlberg} and semilinear equations in \cite{nicola r}. In all these papers the definition of $WF_G(u)$ is modified by replacing the $\cS$-decay in \eqref{Tu decay} with other regularity conditions in order to best fit with the features of the Hamiltonian. In particular, in \cite{CNR survey,CNR prop,CN pot mod} the authors reconsider $WF_G(u)$ in the framework of weighted modulation spaces $M^p$ introduced by Feichtinger, see \cite{fei} and \cite{gro book}. In this connection we address to the next sections, where we shall present an alternative definition in terms of Gabor frames. 

In \cite{nicola r}, to study the non-linear properties of $WF_G(u)$, attention is addressed to $M^2=L^2$ regularity with weight $\la z \ra^s$, $z=(x,\xi)\in \rnn$, corresponding to the spaces $Q^s$ of Shubin \cite{shubin}. In \cite{schulz wahlberg} the authors consider the same variant of $WF_G(u)$, under the action of localization operators. In \cite{wahlberg} the \textit{polynomial Gabor wave front set} is defined assuming \eqref{Tu decay} satisfied for a fixed value of $N$. 

In \cite{cappiello sc,car wahl,CNR wp} the $\cS$-decay is replaced by \textit{analytic} and \textit{Gelfand-Shilov} decay. To be precise, $z_0=(x_0,\xi_0)$ does not belong to such wave front sets if there exists a conic neighbourhood $\Gamma_{z_0}$ of $z_0$ in $\rnn$ such that 
\begin{equation}
|Tu(z)| \le Ce^{-\epsilon \la z\ra^r}, \quad z \in \Gamma_{z_0},
\end{equation} for some fixed $r>0$ and positive constants $C$ and $\epsilon$. The case $r=1$ corresponds to the analytic Gabor wave front set. In \cite{boiti} the definition is generalized to ultradifferentiable classes by assuming
\begin{equation}
|Tu(z)| \le Ce^{-\omega(z)}, \quad z \in \Gamma_{z_0},
\end{equation} for a given weight function $\omega(z)$. 

Observe that in \cite{CNR survey} and \cite{CNR wp} the notion of $WF_G(u)$ is generalized to that of Gabor $\Psi$-filter, respectively in the analytic and modulation space setting. This allows one to get rid of the homogeneity assumption on the Hamiltonian. 

The research related to the Gabor wave front set, or other wave front sets from the point of view of time-frequency analysis, is very intensive at present and it is impossible to give complete references. Let us limit to further mention \cite{cappiello con,debr,kost,pil beyond,pil pran,pil toft}.

\section{Gabor wave front set: theory and practice} In this section we focus on the Gabor wave front set $WF_G$ introduced in the preceding historical account. 

\subsection{The global wave front set of H\"ormander} We briefly review the main properties of the global wave front set $WF(u)$ introduced by H\"ormander in \cite{hormander quadratic}. We need to introduce some preparatory notions. 

\begin{definition}
	The conic support of $a \in \cS'(\rnn)$ is the set $\cones(a)$ of all $z \in \rnno$ such that any open conic neighbourhood $\Gamma_z$ of $z$ in $\rnno$ satisfies: 
	\[ \overline{\supp(a) \cap \Gamma_z}\quad\text{is not compact in }\rnn. \]
\end{definition}

\begin{definition} Let $a \in G^m$ for some $m \in \bR$. We say that a point $z_0 \in \rnno$ is \textit{non-characteristic} for $a$ if there exist positive constants $A,\epsilon> 0$ and an open conic set $\Gamma \subset \rnno$ such that 
	\[ |a(z)| \ge \epsilon \la z \ra^m, \quad z \in \Gamma, \, |z| \ge A. \]
We define $\chr(a)$ as the subset of $\rnno$ containing all the non-characteristic points for $a$.
\end{definition}

Notice that \[ \cones(a) \cup \chr(a) = \rnno, \quad a \in G^m. \] We are now ready to define the global wave front set.

\begin{definition}
	Let $u \in \cS'(\rn)$. We say that a point $z_0 \in \rnno$ does not belong to $WF(u)$ if there exist $m\in \bR$ and $a \in G^m$ such that $a^\w u \in \cS(\rn)$ and $z_0 \notin \chr(a)$. 
\end{definition} 

We collect below some properties satisfied by $WF(u)$, following \cite{rodino wahlberg}.

\begin{proposition} Let $u \in \cS'(\rn)$.
	\begin{enumerate} 
		\item $WF(u)$ is a closed conic subset of $\rnno$. 
		\item $WF(u)$ is symplectically invariant: 
		\[ z_0 \in WF(u) \Rightarrow Sz_0 \in WF(\mu(S)u), \quad S \in \Spdr. \]
		\item For $a\in G^m$ the following inclusions hold:
		\[ WF(a^\w u) \subseteq WF(u) \cap \cones(a) \subseteq WF(u) \subseteq WF(a^\w u) \cup \chr(a). \] In particular, if $\chr(a) = \emptyset$ then $WF(a^\w u) = WF(u)$. 
		\item If $a \in G^m$ and $\cones(a)\cap WF(u) = \emptyset$ then $a^\w u \in \cS(\rn)$. 
		\item $WF(u) = \emptyset$ if and only if $u \in \cS(\rn)$. 
	\end{enumerate}
\end{proposition}

\subsection{The Gabor wave front set at Schwartz regularity} Let us give a concise review of Gabor wave front set in the context of Schwartz regularity, following \cite{rodino wahlberg}. First we introduce a continuous version of the Gabor wave front set characterized by rapid decay of the phase space representation of a distribution. 

\begin{definition} Let $u\in \cS'(\rn)$ and $\vp \in \cS(\rn)\smo$. We say that $z_0 \in \rnno$ does not belong to the set $WF'(u)$ if there exists an open conic neighbourhood $\Gamma_{z_0}$ of $z_0$ in $\rnno$ such that 
	\begin{equation}\label{def WF'} 
	\sup_{z \in \Gamma_{z_0}} \la z \ra^N |V_\vp u(z)| < \infty \quad \forall N \in \bN_0.
	\end{equation}
\end{definition}

It is a direct consequence of the definition that $WF'(u)$ is a closed conic subset of $\rnno$. The definition of $WF'(u)$ is well-posed in the sense that the Schwartz decay of $V_\vp u$ in a conic neighbourhood does not depend on the window function $\vp$, as detailed below. 

\begin{proposition}[{\cite[Cor. 3.3]{rodino wahlberg}}]\label{prop WFG}
	Let $u \in \cS'(\rn)$, $\vp \in \cS(\rn)\smo$ and $z_0 \in \rnno$. Assume that there exists an open conic neighbourhood $\Gamma_{z_0}$ of $z_0$ in $\rnno$ such that condition \eqref{def WF'} holds. For any open conic neighbourhood $\Gamma_{z_0}'$ of $z_0$ in $\rnno$ such that $\overline{\Gamma_{z_0}' \cap \bS^{2n-1}} \subseteq \Gamma_{z_0}$ and any $\psi \in \cS(\rn)\smo$ we have
	\[ \sup_{z \in \Gamma_{z_0}'} \la z \ra^N|V_\psi u(z)| < \infty \quad \forall N \in \bN_0. \]
\end{proposition}

In the spirit of time-frequency analysis it is interesting to study the discrete variant of $WF'(u)$ obtained by replacing the full phase-space cone $\Gamma_{z_0}$ in \eqref{def WF'} with its restriction to suitable lattice points. This leads to the definition of the Gabor wave front set $WF_G(u)$.

\begin{definition} Let $u\in \cS'(\rn)$, $\vp \in \cS(\rn)\smo$ and a separable lattice $\Lambda = \alpha\zn \times \beta\zn$ where $\alpha,\beta>0$ are such that $\cG(\vp,\Lambda)$ is a Gabor frame. We say that $z_0 \in \rnno$ does not belong to the Gabor wave front set $WF_G(u)$ if there exists an open conic neighbourhood $\Gamma_{z_0}$ of $z_0$ in $\rnno$ such that 
	\begin{equation}\label{def WFG} 
	\sup_{\lambda \in \Lambda \cap \Gamma_{z_0}} \la \lambda\ra^N|V_\vp u(z)| < \infty \quad \forall N \in \bN_0.
	\end{equation}
\end{definition}
	
While it is clear that $WF_G(u) \subseteq WF'(u)$, it is a remarkable result that the other inclusion holds too, cf. \cite[Thm. 3.5]{rodino wahlberg}, that is
\begin{equation} \boxed{WF_G(u) = WF'(u), \quad u \in \cS'(\rn).} \end{equation}
This characterization also shows that the definition of $WF_G(u)$ is independent of the choice of the Gabor frame $\cG(\vp,\Lambda)$ used in \eqref{def WFG}. Moreover, it can be proved that all these results still hold for more general lattices $\Lambda= \cA\znn$, where $\cA \in \GLL$. In the rest of the paper we will discard the notation $WF'(u)$ and we compute $WF_G(u)$ according to \eqref{def WF'} whenever convenient.  

Another important achievement in \cite{rodino wahlberg} is the proof of the fact that the Gabor wave front set coincides with H\"ormander's global wave front set. We prefer not to include a discussion of this issue in order to keep the presentation at an introductory level. We just mention that a key ingredient in the proof is a precise characterization of the Gabor wave front set of Weyl operators with symbols in $\Sm$ classes. 

\begin{proposition}\label{prop WFG weyl} Let $m \in \bR$. For $a \in \Sm$ we have
	\[ WF_G(a^\w u) \subseteq \cones(a), \quad u \in \cS'(\rn). \] In particular, for $m=0$ we have 
	\[ WF_G(a^\w u) \subseteq WF_G(u) \cap \cones(a), \quad u \in \cS'(\rn). \]
\end{proposition}

We determine below the Gabor wave front set of some special distributions in order to get a taste of this notion and also to prepare material for applications to Schr\"odinger equations.

\begin{example} Fix $z_0 = (x_0,\xi_0) \in \rnn$. The Gabor wave front set is invariant under time-frequency shifts, namely 
	\[ WF_G(\pi(z_0) u ) = WF_G(u), \quad u \in \cS'(\rn). \] 
	This is indeed a consequence of the invertibility of time-frequency shifts and Proposition \ref{prop WFG weyl}, since 
	\[  \pi(z_0) = \sigma^\w, \quad \sigma(x,\xi) =e^{\pi i x_0\xi_0} e^{2 \pi i(x\xi_0 - \xi x_0)} \in \SO. \]
\end{example}

\begin{example}[Dirac delta]\label{ex delta}
	Consider the Dirac distribution centered at $x_0 \in \rn$, namely $\delta_{x_0} \in \cS'(\rn)$. In view of the previous example we can assume $x_0 = 0$ without loss of generality, namely $WF_G(\delta_{x_0}) = WF_G(\delta_0)$ for all $x_0 \in \rn$. Let us compute the STFT of $\delta_0$: for a fixed window $\vp \in \cS(\rn)\smo$, 
	\[ V_\vp \delta_0(x,\xi) = \la \delta_0, M_\xi T_x \vp \ra = \overline{\vp(-x)}.   \] This implies that $|V_\vp \delta_0 (0,\lambda\xi)| = |\vp(0)|$ for all $\lambda>0$ and $\xi \in \rn$. If we further assume $\vp(0) \ne 0$ we see that $\{0\} \times (\rno) \subseteq WF_G(\delta_0)$. 
	To conclude, let $C>0$ and consider the conic subset $\Gamma = \{ (x,\xi) \in \rnno \,:\, |\xi|<C|x| \}$. Let $z_0 = (x_0,\xi_0) \in \Gamma$; then
	\[ \sup_{z \in \Gamma}\, \la z \ra^N |V_\vp \delta_0 (z)| \lesssim  \sup_{x \in \rn} \, \la x \ra^N |\vp(-x)| < \infty, \] hence $z_0 \notin WF_G(\delta_0)$. This argument allows us to conclude that 
	\[ \boxed{WF_G(\delta_{x_0}) = WF_G(\delta_0) = \{0\} \times (\rno).} \]
	We remark that in the case of $\delta_{x_0}$ the Gabor wave front set is less informative than the classical H\"ormander wave front set \cite{hormander book 1} , which reads $WF_H(\delta_{x_0}) = \{x_0\} \times(\rno)$ and coincides with the SG wave front set $WF_\cS$ by Coriasco and Maniccia \cite{cor man}. 
\end{example}

\begin{example}[Pure frequency]\label{ex freq}
Fix $\xi_0 \in \rn$ and consider the distribution $u(t) = e^{2\pi i t\xi_0}$. In order to determine its Gabor wave front set we apply again the invariance property under phase-space shifts, namely \[ WF_G(u) = WF_G(M_{\xi_0} 1 ) = WF_G(1). \]
For a fixed window $\vp \in \cS(\rn)\smo$ we have
\[ V_\vp 1 (x,\xi) = \la 1,M_\xi T_x \vp\ra = \la \delta_0, T_\xi M_{-x} \hat{\vp} \ra = e^{-2\pi i x \xi} \hat{\vp}(-\xi), \] hence $|V_\vp 1 (\lambda x, 0)| = |\hat{\vp}(0)|$ for any $\lambda >0$ and $x \in \rn$. It is not restrictive to assume $\hat{vp}(0) \ne 0$, thus we conclude $(\rno) \times \{0\} \subseteq WF_G(1)$. The same arguments used in Example \ref{ex delta} yield
\[ \boxed{ WF_G\lc e^{2\pi i \xi_0 \cdot}\rc = WF_G(1) = (\rno)\times \{0\}.} \]
To compare with other wave front sets, notice that the classical wave front set is not able to detect any singularity since $u \in C^{\infty}(\rn)$, hence $WF_H(u) = \emptyset$. However, the SG wave front set is again more precise, yielding $WF_\cS(u)=(\rno)\times \{\xi_0\}$. 
\end{example}

\begin{example}[Fresnel chirp]\label{ex chirp}
	Fix $c \in \bR\smo$ and consider the linear chirp (also known as Fresnel function) $u(t)=e^{\pi i c t^2}$. Straightforward computation for the STFT of $u$ with Gaussian window $\vp(t) = e^{-\pi t^2}$ (cf. for instance \cite{ben unimod}) provide
	\begin{equation}\label{stft gauss} |V_\vp u(x,\xi)| = (1+c^2)^{-n/4}e^{-\pi|\xi-cx|^2/(1+c^2)}. \end{equation}
We deduce that the STFT rapidly decays in any open cone in $\rnno$ which does not include the hyperplane $\xi = cx$. Arguing as in the previous example we conclude that 
\[ \boxed{WF_G \lc e^{\pi i c |\cdot|^2}\rc = \{(x,cx) \, : \, x \in \rno \}.} \]
We stress that the Gabor wave front set is superior in detecting singularities than other notions in this case, which is characterized by varying frequency. Notice indeed that $WF_H(u) = \emptyset$, while $WF_\cS (u)= (\rno)\times(\rno)$.
\end{example}

\subsection{Modulation space setting}
	 
In Section \ref{sec mod} we introduced modulation spaces by conditioning the (weighted and mixed) Lebesgue regularity of the phase-space representation (STFT) of their members. This notion suggests a natural generalization of the Gabor wave front set $WF_G$ by relaxing the Schwartz decay in \eqref{def WF'} as follows, cf. \cite{CNR prop}.

\begin{definition}\label{def WFG mod cont}
	Let $1\le p \le \infty$, $s\ge0$, $\vp \in \cS(\rn)\smo$ and $u \in \cS'(\rn)$. We say that $z_0 \in \rnno$ does not belong to $WF_G^{p,s}(u)$ if there exists an open conic neighbourhood $\Gamma_{z_0}$ of $z_0$ in $\rnno$ such that $V_\vp u \in L^p\s(\Gamma_{z_0})$, that is
	\begin{equation}\label{def mod WFG} \int_{\Gamma_{z_0}} |V_\vp u(z)|^p \la z \ra ^{sp} dz < \infty, \end{equation}
	with obvious modification in the case where $p=\infty$. 
\end{definition}
It is clear from the definition that $WF_G^{p,s}(u)$ is a closed conic subset of $\rnno$. 

We remark that other kinds of microlocal analysis at modulation space regularity may be taken into account. In this respect we mention the wave front set $WF_{M^{p,q}_m}(u)$ introduced in \cite{pil tt1,pil tt2} and defined as follows. First define for $f\in \cS'(\rn)$ the set $\Sigma(f)$ as the complement in $\rno$ of the subset which contains all $\bar{\xi} \in \rno$ such that 
\[ \lc \int_{\Gamma_{\bar{\xi}}} \lc \int_{\rn} \left| V_\vp f (x,\xi) \right|^p m(x,\xi)^p dx \rc^{q/p} d\xi \rc^{1/q} < \infty, \] for some conic neighbourhood $\Gamma_{\bar{\xi}}$ of $\bar{\xi}$ in $\rno$. Hence, for $1\le p,q \le \infty$, a weight function $m$ on $\rnn$ and $u \in \cD'(\Omega)$, $\Omega \subseteq \rn$ open, $WF_{M^{p,q}_m}(u)$ consists of elements $(x_0,\xi_0) \in \Omega \times \rno$ such that $\xi_0 \in \Sigma(\phi u)$ for any $\phi \in C^{\infty}_c(\Omega)$ with $\phi(x_0) \ne 0$. It is a remarkable result that modulation spaces are microlocally equivalent to Fourier-Lebesgue spaces, in the sense of \cite[Thm. 6.1]{pil tt1}. We also refer to \cite{johansson} for a discrete version of this analysis. 

We prove below the independence of the window $\vp$ in the definition of $WF_G^{p,s}$, cf. \cite{CNR prop} for more general results.  
\begin{proposition}\label{prop indep win WFGM}
	Let $1 \le p \le \infty$, $s\ge0$, $u \in \cS'(\rn)$, $\vp \in \cS(\rn)\smo$ and $z_0 \in \rnno$. Assume that there exists an open conic neighbourhood $\Gamma_{z_0}$ of $z_0$ in $\rnno$ such that condition \eqref{def mod WFG} holds. For any open conic neighbourhood $\Gamma_{z_0}'$ of $z_0$ in $\rnno$ such that $\overline{\Gamma_{z_0}' \cap \bS^{2n-1}} \subseteq \Gamma_{z_0}$ and any $\psi \in \cS(\rn)\smo$ we have
	\begin{equation}\label{change wind WFGM} \int_{\Gamma_{z_0}'} |V_\psi u(z)|^p \la z \ra ^{sp} dz < \infty. \end{equation}
\end{proposition}
\begin{proof} Let us first recall the change-of-window estimate in Lemma \ref{change of w}, namely 
	\[ |V_{\psi}u(z)| \lesssim (|V_{\vp}u| * |V_{\psi}\vp|)(z), \quad z\in \rnn. \] Since $V_{\psi}\vp \in \cS(\rnn)$ for $\psi,\vp \in \cS(\rn)$, for any $N\ge0$ we have
	\[ |V_{\psi}u(z)| \lesssim \int_{\rnn} \la z-w \ra^{-N}|V_{\vp}u(w)| dw. \] 
	Therefore, to prove the desired estimate \eqref{change wind WFGM} it is enough to show that, for a suitable choice of $n\ge 0$ we have 
	\[ \lV \int_{\rnn} F(\cdot,w)dw \rV_{L^p(\Gamma_{z_0}')} < \infty, \] where we set $F(z,w) = F_n(z,w) = \la z \ra^s \la z-w \ra^{-N}|V_{\vp}u(w)|$. \\	
	We conveniently split the domain of integration in $\int_{\rnn} F(\cdot,w)dw$ in two parts, namely $\Gamma_{z_0}$ and $\rnn \setminus \Gamma_{z_0}$. Let us first consider $\rnn \setminus \Gamma_{z_0}$ and notice that \begin{equation}\label{est cone} \la z-w \ra \gtrsim \max\{ \la z\ra, \la w\ra \}, \quad z \in \Gamma_{z_0}', \, w \in \rnn \setminus \Gamma_{z_0}. \end{equation} Furthermore, in view of the characterization of $\cS'(\rn)$ in \eqref{eq cS mod} we deduce that $u \in M^p_{v_{-r}}(\rn)$ for some $r\ge 0$. Therefore, for $z \in \Gamma_{z_0}'$ we may write
\begin{align*} \int_{\rnn\setminus \Gamma_{z_0}} F(z,w)dw & \le \int_{\rnn\setminus \Gamma_{z_0}} \la z \ra^s\la w\ra^r \la z-w \ra^{-N} \frac{|V_{\vp}u(w)|}{\la w \ra^r}dw \\ & \lesssim \lc \la \cdot \ra^{r+s-N} * \frac{|V_{\vp}u(\cdot)|}{\la \cdot \ra^r} \rc(z). 
\end{align*} It is then enough to assume $N > r+s+2n$ to conclude
\[ \lV \int_{\rnn\setminus \Gamma_{z_0}} F(\cdot,w)dw \rV_{L^p(\Gamma_{z_0}')} \lesssim \lV \la \cdot \ra^{r+s-N} \rV_{L^1(\rnn)} \lV u \rV_{M^p_{v_{-r}(\rn)}} < \infty.  \]
For the remaining part we have
\begin{align*} \int_{\Gamma_{z_0}} F(z,w)dw & \le \int_{ \Gamma_{z_0}} \la z \ra^s\la w\ra^{-s} \la z-w \ra^{-s} \la z-w \ra^{s-N} |V_{\vp}u(w)\la w \ra^sdw \\ & \lesssim \int_{ \Gamma_{z_0}} \la z-w \ra^{s-N} |V_{\vp}u(w)\la w \ra^sdw \\ & \lesssim \lc \la \cdot \ra^{s-N} * \lc 1_{\Gamma_{z_0}}(\cdot)|V_{\vp}u(\cdot)|\la \cdot \ra^s \rc \rc(z),
\end{align*} where $1_{\Gamma_{z_0}}$ is the characteristic function of the set $\Gamma_{z_0}$. Assumption \eqref{def mod WFG} finally yields
\[ \lV \int_{\Gamma_{z_0}} F(\cdot,w)dw \rV_{L^p(\Gamma_{z_0}')} \lesssim \lV \la \cdot \ra^{s-N} \rV_{L^1(\rnn)} \lV V_{\vp}u \rV_{L^p_{v_s}(\Gamma_{z_0})} < \infty. \]
\end{proof}

%{\color{red} \begin{remark} We remark that less regular window functions may be considered in \eqref{def mod WFG}, in particular at the modulation space level provided by $M^{\infty}_{v_r}(\rn)$ for some $r \in \bR$. To be precise, for $1 \le p \le \infty$, $r,s \in \bR$ with $r>s+2n$, given $u \in M^1_{v_{-r}}(\rn)$ we say that $z_0 \notin WF_G^{p,s}(u)$ for $s\ge0$ if there exist an open conic neighbourhood $\Gamma_{z_0}$ of $z_0$ in $\rnno$ and some $\vp \in M^{\infty}_{v_r}(\rn)\smo$ with $r>s+2n$ such that $V_{\vp}u \in L^p_{v_s}(\Gamma_{z_0})$. The previous proof still works, hence if $z_0 \notin WF_G^{p,s}(u)$ then $V_{\vp}u \in L^p_{v_s}(\Gamma_{z_0}) \Rightarrow V_{\psi}u \in L^p_{v_s}(\Gamma_{z_0}')$ for any $\psi \in  M^{\infty}_{v_N}(\rn)\smo$, $N>r+s+2n$, and $\Gamma_{z_0}'$ as in Proposition \ref{prop indep win WFGM}. \end{remark} }

In complete analogy with the Gabor wave front set $WF_G$ introduced in Definition \ref{def WFG} we consider a discrete version of $WF_G^{p,s}$. 

\begin{definition}\label{def WFG mod disc}
	Let $1\le p \le \infty$, $s\ge0$, $\vp \in \cS(\rn)\smo$ and $u \in \cS'(\rn)$. Consider a separable lattice $\Lambda = \alpha \zn \times \beta \zn$ where $\alpha,\beta>0$ are such that $\cG(\vp,\Lambda)$ is a Gabor frame. We say that $z_0 \in \rnno$ does not belong to $\wt{WF_G^{p,s}}(u)$ if there exists an open conic neighbourhood $\Gamma_{z_0}$ of $z_0$ in $\rnno$ such that $V_\vp u \in L^p\s(\Gamma_{z_0})$, that is
	 \begin{equation}\label{def mod WFG disc} \sum_{\lambda \in \Lambda \cap \Gamma_{z_0}} |V_\vp u(\lambda)|^p \la \lambda \ra ^{sp} < \infty, \end{equation}
	  with obvious modification in the case where $p=\infty$. 
  \end{definition}

We show that the discrete and continuous modulation Gabor wave front set coincide. Therefore, modulation space regularity in a conic neighbourhood of a phase space direction is a condition as strong as modulation space regularity restricted to the points of the same cone which belong to a suitable lattice. 
  
\begin{theorem}
	Let $1\le p \le \infty$, $s\ge0$ and $u \in \cS'(\rn)$. Then $WF_G^{p,s}(u) = \wt{WF_G^{p,s}}(u)$. 
\end{theorem}
\begin{proof} We give the proof only in the case where $p<\infty$, since the case $p=\infty$ requires trivial modification. We first prove that $z_0 \notin \wt{WF_G^{p,s}}(u) WF_G^{p,s}(u)$, namely that \eqref{def WFG mod disc} implies \eqref{def WFG mod cont}. In view of the reconstruction formula \eqref{rec form disc} we write $u = u_1 + u_2$, where 
	\[ u_1 = \sumin V_\vp u (\lambda) \pi(\lambda)\vpt, \qquad u_2 =\sumo V_\vp u (\lambda) \pi(\lambda)\vpt, \] where $\vpt = S^{-1}\vp \in \cS(\rn)\smo$ is the canonical dual window. 
 	It is therefore enough to show that $V_\vp u_1,V_\vp u_2 \in L^p_{v_s}(\Gamma_{z_0})$. Let us start with $V_\vp u_1$. 
 	\begin{align*}
 	\lV V_\vp u_1 \rV^p_{L^p_{v_s}(\Gamma_{z_0})} & = \int_{\Gz} \left| V_\vp u_1 \right|^p \la z \ra^{ps} dz \\ & \le \int_{\Gz} \sumin \lc \left| V_\vp u(\lambda) \right| \left| V_{\vpt}\vp(z-\lambda) \right| \la z \ra^{s}\rc^p dz.
 	\end{align*}
 We use the subadditivity of the weight, namely the identity $\la z \ra^s \le \la z-\lambda \ra^{-s} \la \lambda \ra^s$ to get
 \begin{align*}
 \lV V_\vp u_1 \rV^p_{L^p_{v_s}(\Gamma_{z_0})} & \le \int_{\Gz} \sumin \lc \left| V_\vp u(\lambda) \right| \la \lambda \ra^s \left| V_{\vpt}\vp(z-\lambda) \right| \la z-\lambda \ra^{-s}\rc^p dz.
 \end{align*}
Let us set $f(\lambda) = \left| V_\vp u(\lambda) \right| \la \lambda \ra^s$ and $g(z-\lambda) = \left| V_{\vpt}\vp(z-\lambda) \right| \la z-\lambda \ra^{-s}$ for the sake of clarity. Notice that $g(z- \lambda) \lesssim \la z-\lambda \ra^{-N-s}$ for arbitrary $N \ge 0$. Hence, by H\"older inequality we have
 \begin{align*}
\lV V_\vp u_1 \rV^p_{L^p_{v_s}(\Gamma_{z_0})} & \le \int_{\Gz} \sumin \lc f(\lambda)g(z-\lambda)\rc^p dz \\
& = \int_{\Gz} \sumin \lc f(\lambda)g(z-\lambda)^{1/p}g(z-\lambda)^{1-1/p}\rc^p dz \\
& \le \int_{\Gz} \lc \sumin f(\lambda)^p g(z-\lambda)\rc \lc \sumin g(z-\lambda)\rc^{p/p'} dz \\ 
& \le C \int_{\Gz} \sumin f(\lambda)^p g(z-\lambda)dz,
\end{align*}
where\[ C = \sup_{z \in \rnn} \lV g(z-\cdot) \rV_{\ell^1}^{p/p'} < \infty. \] We conclude by Minkowski inequality:
 \begin{align*}
\lV V_\vp u_1 \rV^p_{L^p_{v_s}(\Gamma_{z_0})} & \le C \int_{\Gz} \sumin f(\lambda)^p g(z-\lambda)dz \\
& \le C \sumin f(\lambda)^p \int_{\Gz} g(z-\lambda) dz \\
& \le C' \sumin f(\lambda)^p < \infty,
\end{align*}
where we set 
\[ C' = C \int_{\Gz} g(z - \lambda)dz < \infty, \] and used the assumption \eqref{def WFG mod disc} in the last step. 

It remains to prove that $V_\vp u_2 \in L^p_{v_s}(\Gz)$, namely 
 \begin{align*}
\lV V_\vp u_2 \rV^p_{L^p_{v_s}(\Gamma_{z_0})} & = \int_{\Gz} \left| V_\vp u_2 \right|^p \la z \ra^{ps} dz \\ & \le \int_{\Gz} \sumo \lc \left| V_\vp u(\lambda) \right| \left| V_{\vpt}\vp(z-\lambda) \right| \la z \ra^{s}\rc^p dz.
\end{align*}
Recall from Section \ref{sec mod} that the STFT has at most polynomial growth, that is $|V_\vp u(\lambda)| \lesssim \la \lambda \ra^r$ for some $r \ge 0$. Moreover, since $V_{\vpt}\vp \in \cS(\rnn)$ we have $|V_{\vpt}\vp(z-\lambda)|\lesssim \la z-\lambda\ra^{-N}$ for any $N \ge 0$. As a consequence of \eqref{est cone} we have 
\begin{align*}
\lV V_\vp u_2 \rV^p_{L^p_{v_s}(\Gamma_{z_0})} & \le \int_{\Gz} \sumo \lc \left| V_\vp u(\lambda) \right| \left| V_{\vpt}\vp(z-\lambda) \right| \la z \ra^{s}\rc^p dz \\
& \lesssim \int_{\Gz} \sumo \lc \la \lambda \ra^r \la z-\lambda \ra^{-N} \la z \ra^s \rc^p dz \\
& \lesssim \int_{\Gz} \sumo \lc \la \lambda \ra^{r-N/2} \la z \ra^{s-N/2} \rc^p dz \\
& \lc \int_{\Gz} \la z \ra^{p(s-N/2)}\rc \lc \sumo \la \lambda \ra^{p(r-N/2)} \rc < \infty,
\end{align*}
where the conclusion follows after choosing $N$ large enough. 

We need to prove now that $z_0 \notin WF_G^{p,s}(u) \Rightarrow z_0 \notin \wt{WF_G^{p,s}(u)}$, that is \eqref{def WFG mod cont} implies \eqref{def WFG mod disc}. We essentially argue as before after inverting the role of discrete and continuous norms and reconstruction formulae. To be concrete we prove that $V_\vp u \in \ell^p_{v_s}(\Lambda \cap \Gz)$. In view of the inversion formula for the STFT in \eqref{rec form cont} we set $u=u_1' + u_2'$, where
\[ u_1' = \int_{\Gz} V_\vp u(z)\pi(z)\vp dz, \quad u_2' = \int_{\rnn\setminus \Gz} V_\vp u(z)\pi(z)\vp dz. \] 
It is enough to prove that $V_\vp u_1', V_\vp u_2' \in \ell^p_{v_s}(\Lambda \cap \Gz)$. Let us first prove the claim for $V_\vp u_1'$, having in mind \eqref{stft rec cont}. We have
\begin{align*}
\lV V_\vp u_1' \rV_{\ell^p_{v_s}(\Lambda \cap \Gz)} & = \sumin \left| V_\vp u_1 (\lambda)\right|^p \la \lambda \ra^{sp} \\
& \lesssim \sumin \lc \int_{\Gz} \left| V_\vp u (z) \right| \left|V_\vp \vp (\lambda-z) \right| dz \rc^p \\
& \lesssim \sumin \lc \int_{\Gz} \left|V_\vp u(z) \right| \la z \ra^s \left|V_\vp \vp(\lambda-z)\right| \la \lambda-z\ra^{-s} dz\rc^p. 
\end{align*} 
We set $f(z) = \left|V_\vp u(z) \right| \la z \ra^s$ and $h(\lambda-z) = \left|V_\vp \vp(\lambda-z)\right| \la \lambda-z\ra^{-s}$ in order to lighten the notation. Therefore, by applying again H\"older's inequality we get
\begin{align*}
\lV V_\vp u_1' \rV_{\ell^p_{v_s}(\Lambda \cap \Gz)} & \lesssim \sumin \lc \int_{\Gz} f(z)h(\lambda-z) dz\rc^p \\
& \le \sumin \lc \int_{\Gz} f(z)^p h(\lambda-z) dz\rc \lc \int_{\Gz} h(z-\lambda) dz\rc^{p/p'}\\
& \le \lV h \rV_{L^1}^{p/p'} \sumin \int_{\Gz} f(z)^p h(\lambda-z) dz \\
& \le C \int_{\Gz} f(z)^p dz < \infty,
\end{align*} 
where we used the assumption \eqref{def WFG mod cont} in the last step and we set 
\[ C= \lV h \rV_{L^1}^{p/p'} \sup_{z \in \rn} \sumin h(z-\lambda) < \infty. \] The proof of $V_\vp u_2' \in \ell^p_{v_s}(\Lambda \cap \Gz)$ follows the same pattern of the proof of $V_\vp u_2 \in L^p_{v_s}(\Gz)$ above, hence is left to the interested reader. 
 \end{proof}

\begin{remark} As a consequence of the previous identification and Proposition \ref{prop indep win WFGM} we have that $\wt{WF_G^{p,s}(u)}$ does not depend on the Gabor frame $\cG(\vp,\Lambda)$ used in \eqref{def WFG mod disc}. Moreover, it is clear from the definition that $u \in M^p_{v_s}(\rn)$ if and only if $WF_G^{p,s}(u) = \emptyset$, in view of the compactness of the sphere $\bS^{2n-1}$. 
\end{remark}

The modulation space Gabor wave front set is very well suited to the study of Weyl operators with low regular symbols, as detailed in the following result. 

\begin{proposition}[{\cite[Prop. 5.3]{CNR prop}}] Let $1\le p \le \infty$, $a \in M^{\infty}_{1\otimes v_\gamma}(\rnn)$ with $\gamma > 2n$ and $0 < 2s < \gamma - 2n$. For any $u \in M^p_{-s}(\rn)$ we have 
	\[ WF^{p,s}_G(a^\w u) \subset WF^{p,s}_G(u). \]
\end{proposition}

This should be compared with Proposition \ref{prop WFG}, having in mind that $\bigcap_{\gamma \ge 0} M^{\infty}_{1\otimes v_\gamma}(\rnn) = S^0_{0,0}$.

\subsection{Propagation of singularities} 
We conclude this survey with some easy examples of application of the Gabor wave front set to propagation of microlocal singularities for Schr\"odinger equations. We refer to \cite{CNR prop,nicola r} for a broader treatment of the topic, see also the other references cited in the historical account above. 

Let us fix the setting of our investigation. We consider the Cauchy problem for the Schr\"odinger equation, namely 
\begin{equation}\label{eq schro}
\begin{cases} i\partial_t u(t,x) = Hu(t,x) \\ u(0,x) = u_0(x) \end{cases}, 
\end{equation}
where $H=a^\w$ is the Weyl quantization of a real-valued quadratic polynomial in $\rnn$, namely
\begin{equation}\label{quad symb} a(x,\xi) = \frac{1}{2}xAx + \xi B x + \frac{1}{2}\xi C \xi, \end{equation} for some symmetric matrices $A,C\in \bR^{n\times n}$ and $B\in \bR^{n\times n}$. The phase-space analysis of the Schr\"odinger propagator $U(t) : u_0(x) \mapsto u(t,x)$ is intimately related to the corresponding Hamiltonian system, that is\footnote{The factor $2\pi$ is a consequence of the normalization of the Fourier transform adopted in this paper.} \[ 2\pi \dot{z} = J \nabla_z a(z) = \mathbb{A}, \quad \mathbb{A}= \left(\begin{array}{cc} B & C \\ -A & -B^{\top}\end{array}\right). \]
The classical phase-space flow $\cA_t=e^{(t/2\pi)\mathbb{A}}: \rnn \to \rnn$ is a symplectic diffeomorphism and the following result on the propagation of singularities holds in our setting. 
\begin{theorem}
	Consider the Cauchy problem \eqref{eq schro} with the assumption specified above. We have that $U(t) \in \cB(M^p_{v_{r}}(\rn))$ for all $t \in \bR$, $1 \le p \le \infty$ and $r\in \bR$. If $u_0 \in \cS'(\rn)$ then
	\[ WF_G(U(t)u_0) = \cA_t (WF_G(u_0)), \quad t \in \bR. \] If in particular $u_0 \in M^p_{v_{-s}}(\rn)$ for some $1\le p \le \infty$ and $s\ge 0$ then 
		\[ WF_G^{p,s}(U(t)u_0) = \cA_t (WF_G^{p,s}(u_0)), \quad  t \in \bR. \]
\end{theorem}

More refined results for general Hamiltonians and potential perturbations can be found in \cite{CNR prop}. We stress that this is one of the rare case where propagation of singularities for Schr\"odinger operators with non-smooth potentials is taken into account. 

\begin{example}[The free particle] Let us first consider the free case, namely $H= - \triangle/2$ - which corresponds to $a(x,\xi)=\xi^2/2$. It is well known that the solution of \eqref{eq schro} can then be expressed as 
	\[ u(t,x) = (K_t * u_0)(x), \quad K_t(x)=\frac{1}{(2\pi it)^{n/2}}e^{ix^2/(2t)}. \]
An easy computation reveals that the corresponding Hamiltonian flow is given by 
\[ \cA_t(x,\xi) = (x-2\pi t\xi,\xi), \quad (x,\xi)\in \rnn.\]
Let us consider the initial datum $u_0 = \delta_0$, so that $U(t)u_0(x) = K_t(x)$. Therefore, using the results in Example \ref{ex delta} we get
\[ WF_G(U(t)u_0) = \cA_t(WF_G(\delta_0)) = \{ (x,\xi) \in \rnn : x=2\pi t \xi, \xi \ne 0\}.\] Notice that a pure frequency initial state, namely $u_0(x) = e^{2\pi i x\xi_0}$ for $\xi_0 \in \rn$, evolves as 
$u(t,x) = e^{-2\pi i \xi_0^2}e^{2\pi i x \xi_0}$, hence the wave front set is stationary: 
\[ WF_G(U(t)u_0) = WF_G(u_0) = \{(x,0) \in \rnn, x \ne 0\}. \]
\end{example}

\begin{example}[The harmonic oscillator] 
	Consider now the Hamiltonian 
	\[ H = -\frac{1}{4\pi}\triangle + \pi x^2, \] that is the Weyl quantization of the symbol $a(x,\xi)$ as in \eqref{quad symb} with $A=(2\pi) I$, $B=0$ and $C=-(2\pi)I$, where $I$ is the $n\times n$ identity matrix - see \cite[Sec. 4.3]{folland} and \cite[Sec. 4]{CN pot mod} for a detailed derivation. The classical flow can be explicitly computed:
	\[ \cA_t = e^{(t/2\pi)\mathbb{A}} = \lc \begin{array}{cc} (\cos t)I & (\sin t)I \\ -(\sin t) I & (\cos t)I \end{array}\rc, \quad t \in \bR. \]
	Therefore, by taking into account the initial datum $u_0 = 1$ and Example \ref{ex freq} above we have for any $t \in \bR$
	\[ WF_G(U(t)u_0) = \cA_t (WF_G(1)) = \{ (x,\xi)=((\cos t)y,(\sin t)y)\in \rnn, \, y\ne 0\}. \] Let us examine the behaviour of the wave front set in the interval $t\in [0,\pi/2]$ for the sake of concreteness. For $t=0$ we have $WF_G(u_0) = (\rno)\times \{0\}$, while for $t=\pi/2$ we have $WF(U(\pi/2)u_0) = \{0\} \times (\rno)$. We see that for $t \in (0,\pi/2)$ the singularities are propagated by counter-clockwise rotation in phase space. Let us stress the connection with the structure of the propagator, whose distribution kernel is given by the \textit{Mehler formula} \cite{dg symp,kapit}: for $k \in \bZ$,
	\begin{equation}\label{mehler} K_t(x,y) = \begin{cases} c(k) |\sin t|^{-n/2}\exp \left( \pi i \frac{x^2+y^2}{\tan t} - 2\pi i \frac{x y}{\sin t} \right) & (\pi k < t < \pi(k+1)) \\ c'(k)\delta((-1)^k x-y) & (t=k\pi) \end{cases}, \end{equation} for suitable phase factors $c(k),c'(k)\in \bC$. 
	The solution is thus given in the form of Fourier integral operator for $t\ne \pi/2 + k\pi$, $k \in \bZ$, as
	\[ U(t)u_0(x) = (\cos t)^{-n/2} \int_{\rn} e^{2\pi i \left[ \frac{1}{\cos t}x\xi -\frac{\tan t}{2} (x^2+\xi^2)\right]} \widehat{u_0}(\xi)d\xi. \] In particular, the choice $u_0 = 1$ yields  $u(t,x)=U(t)u_0(x) = (\cos t)^{-n/2} e^{-\pi i (\tan t) x^2}$, which is consistent with the previous computation, since by \eqref{stft gauss} we have
	\[ \left|V_\vp u(t,\cdot)(x,\xi)\right| =e^{-\pi (\cos t)^2 |\xi-(\tan t)x|^2}. \]
\end{example}

\end{document}